\documentclass{amsart}
\usepackage{amscd}
\usepackage{amssymb}
\usepackage[all]{xy}
\usepackage{mathrsfs}
\usepackage{upgreek}
\usepackage{stmaryrd}
\usepackage[colorlinks = true]{hyperref}
\usepackage{rotating}
\usepackage{tikz}
\usepackage{tikz-cd}
\usepackage{bbm}
\usetikzlibrary{matrix,arrows,cd,decorations.pathmorphing}

\newcommand{\bk}{\Bbbk}

\newcommand{\Z}{\mathbb{Z}}

\newcommand{\K}{\mathbb{K}}

\renewcommand{\O}{\mathbb{O}}
\newcommand{\F}{\mathbb{F}}

\newcommand{\Gm}{\mathbb{G}_{\mathrm{m}}}

\newcommand{\fg}{\mathfrak{g}}

\newcommand{\ft}{\mathfrak{t}}

\newcommand{\uenv}{\mathcal{U}}

\newcommand{\Dist}{\mathrm{Dist}}

\newcommand{\Diag}{\mathrm{Diag}}

\newcommand{\bX}{\mathbf{X}}

\newcommand{\cG}{\mathcal{G}}

\DeclareMathOperator{\Spec}{Spec}

\newcommand{\id}{\mathrm{id}}

\newcommand{\la}{\langle}
\newcommand{\ra}{\rangle}

\newcommand{\ol}[1]{\overline{#1}}

\makeatletter
\def\lotimes{\@ifnextchar_{\@lotimessub}{\@lotimesnosub}}
\def\@lotimessub_#1{\mathchoice{\mathbin{\mathop{\otimes}^L}_{#1}}%
  {\otimes^L_{#1}}{\otimes^L_{#1}}{\otimes^L_{#1}}}
\def\@lotimesnosub{\mathbin{\mathop{\otimes}^L}}
\makeatother

\numberwithin{equation}{section}
\newtheorem{thm}{Theorem}[section]
\newtheorem{lem}[thm]{Lemma}
\newtheorem{prop}[thm]{Proposition}
\newtheorem{cor}[thm]{Corollary}

\theoremstyle{definition}
\newtheorem{defn}[thm]{Definition}

\theoremstyle{remark}
\newtheorem{rmk}[thm]{Remark}
\newtheorem{ex}[thm]{Example}

\title[On the centralizer of a balanced nilpotent section]{On the centralizer of a balanced nilpotent section}

  \author{William Hardesty}
   \address{Department of Mathematics\\
   Louisiana State University\\
   Baton Rouge, LA 70803\\
   U.S.A.}
  \email{whardesty@lsu.edu}

\begin{document}
\maketitle

  \begin{abstract}
  Let $G$ be a split reductive algebraic group defined over a complete discrete valuation ring $\O$, with residue field $\F$ and fraction field $\K$, where 
  the fiber $G_\F$ is geometrically standard. 
  A balanced nilpotent section $x \in \text{Lie}(G)$ can roughly be thought of as an $\O$-point in a $\K$ nilpotent orbit such that 
  the corresponding orbits over $\K$ and $\F$ have the same Bala--Carter label. 
  In this paper, we will establish a number of results on the structure of the centralizer $G^x \subseteq G$ of $x$.  
  This includes a proof that 
  $G^x$ is a smooth group scheme, and that the component groups of the geometric fibers $G^x_{\ol{\K}}$ and $G^x_{\ol{\F}}$ are isomorphic. 
  \end{abstract}

\section{Introduction}\label{sec:prelim}
Let $p$ be a prime, and let $\O$ be a complete discrete-valuation-ring (DVR) with uniformizer $\omega$,  fraction field $\K$, and perfect residue field $\F$
of characteristic $p$.
For an $\O$-scheme $X$, and an $\O$-algebra $A$, we use the notation
$X_A := X \times_{\Spec(\O)} \Spec(A)$. 
 Let $G$ be a split reductive algebraic group scheme over $\Spec(\O)$, with a (lower) Borel subgroup $B$ and a maximal torus $T$. 
We will assume that $G_\K$ and $G_\F$ are \emph{geometrically standard} (cf. \cite[\S 3.1]{mcninch2016}). 

\begin{ex}
If $G=SL_n$, then $G_\F$ is geometrically standard provided $p \nmid n$. 
\end{ex}

Let $\fg= \text{Lie}(G)$ be the Lie algebra of $G$, regarded as a scheme, and let  $\fg_\O := \fg(\O)$ denote the $\O$-points of $\fg$, which gives
a lattice $\fg_\O \subset \fg_\K$. We will refer to the elements of $\fg_\O$ as \emph{sections}, since such an element corresponds to a map 
$x: \Spec(\O) \rightarrow \fg$. For each section $x \in \fg_\O$, let $x_\K \in \fg_\K$ and $x_\F \in \fg_\F$ denote the values of this map at the generic point and closed point of $\Spec(\O)$
respectively. 

The adjoint action gives $\fg_\O$ the structure of a rational $G$-module (equivalently an $\O[G]$-comodule). For any $x \in \fg_\O$, let 
$G^x \subseteq G$ denote the scheme-theoretic centralizer of $x$,  and let 
$G^{x_\K}_\K \subseteq G_\K$ and $G^{x_\F}_\F \subseteq G_\F$ denote the scheme-theoretic centralizers of $x_\K$ and $x_\F$ (see \cite[I.2.12(1)]{jantzen} for the definition).\footnote{To simplify notation, we will often write $G^x_\K$ and $G^x_\F$
in place of $G^{x_\K}_\K$ and $G^{x_\F}_\F$ respectively.}
 In fact, it can be deduced from the definition that for any $\O$-algebra $A$, 
\[
G^{x_A}_A \cong (G^x)_A. 
\]
To improve notation, we will often omit the parenthesis on the right hand side, and take $G^x_A \subseteq G_A$ to mean the centralizer of $x_A \in \fg_A$.

\begin{defn}\label{defn:balanced-section}
We say that a section $x$ is \emph{balanced} if $G^{x_\K}_\K$ and $G^{x_\F}_\F$ are smooth group schemes and 
$\dim G^{x_\K}_\K = \dim G^{x_\F}_\F$ (cf. \cite{mcninch2016}). 
\end{defn}
\begin{rmk}\label{rmk:smooth-fiber}
The definition of balanced in particular implies that the scheme-theoretic centralizers of $x_\K$ and $x_\F$ are actually reduced, since smooth implies reduced. 
Thus, over the algebraic closures $\ol{\K}$ and $\ol{\F}$, the base changes $G^x_{\ol{\K}}$ and $G^x_{\ol{\F}}$ coincide with the ``classical" centralizers,
which are defined only on the geometric points as in \cite{ls}. 
\end{rmk}
\begin{rmk}\label{rmk:bala-carter}
It follows from \cite[Theorem~4.5.2 and Corollary~7.3.2]{mcninch2016} and the proof of \cite[Corollary~9.2.2]{mcninch2016},  that the orbits of $x_{\ol{\K}}$ and $x_{\ol{\F}}$ have the same Bala--Carter label. 
\end{rmk}

It is worth mentioning that 
that balanced nilpotent sections exist for every $\overline{\F}$-orbit by \cite[Theorem 4.5.2]{mcninch2016}. Thus, since there exist only finitely many 
nilpotent $\ol{\F}$-orbits, it is possible to enlarge $\O$ by a \emph{finite} extension so that for every $\ol{\F}$ nilpotent orbit $C_{\ol{\F}}$, there exists a balanced 
section $x \in \fg_\O$ such that $x_\F \in C_{\ol{\F}}$ (the corresponding extension of $\F$ is also finite, and thus will remain perfect). 


\subsection{Smoothness of $G^x$}
By a \emph{smooth} morphism $f:X\rightarrow Y$ of schemes, we will mean a morphism which satisfies the definition given in 
\cite[\href{https://stacks.math.columbia.edu/tag/01V4}{Tag 01V8}]{stacks-project}. 
We have included a proof of the following lemma due to the
lack of a proper reference. 
\begin{lem}\label{defn:smooth}
A morphism $f:X \rightarrow Y$ between schemes of finite-type is smooth if and only if
\begin{enumerate}
\item $f$ is \emph{flat},
\item for every geometric point $\ol{y} \rightarrow Y$, the fiber product $X \times_Y \ol{y}$ is a smooth variety. 
\end{enumerate}
\end{lem}
\begin{proof}
By first applying  \cite[\href{https://stacks.math.columbia.edu/tag/01V8}{Tag 01V4}]{stacks-project}, we can reduce down to checking 
smoothness at every fiber $X \times_Y y$ for any point $\{y\} \rightarrow  Y$. By definition, $\{y\} = \Spec(\bk)$ for some field $\bk$, and so by 
 \cite[Theorem~III.10.2]{har:ag}, $X \times_Y \Spec(\bk)$ is smooth if and only if $X\otimes_Y \Spec(\ol{\bk})$ is smooth. Combining these two results
 gives the lemma. 
\end{proof}

We will prove the following theorem. 
\begin{thm}\label{thm:main-thm}
If $x \in \fg_\O$ is a balanced nilpotent section, then $G^x \rightarrow \Spec(\O)$ is a smooth morphism. 
 \end{thm}
 \begin{rmk}\label{rmk:main-thm}
We already know that the morphism $G^x \rightarrow \Spec(\O)$ is finite-type with geometric fibers
$G^{x}_{\ol{\K}}$ and $G^{x}_{\ol{\F}}$, which are smooth varieties by Remark~\ref{rmk:smooth-fiber}, so according to Lemma~\ref{defn:smooth}, it suffices 
to show that $\O[G^x]$ is a flat $\O$-module. However, since $\O$ is a DVR, this is equivalent to proving that $\O[G^x]$ is torsion-free over $\O$. 
  This will be proven in \S\ref{subsec:disconnected-case}. 
\end{rmk}

\subsection{Results on component groups}\label{ssec:comp-group-intro}
 Let $A(x_{\ol{\K}}) = G^x_{\ol{\K}}(\ol{\K})/(G^x_{\ol{\K}})^\circ(\ol{\K})$ and $A(x_{\ol{\F}})= G^x_{\ol{\F}}(\ol{\F})/(G^x_{\ol{\F}})^\circ(\ol{\F})$ denote the (discrete) component 
 groups of the geometric fibers $G^x_{\ol{\K}}$ and $G^x_{\ol{\F}}$ respectively. 
 In the case where $G$ is simple and of adjoint type, it follows from Remark~\ref{rmk:bala-carter} and \cite{mcninch-sommers} that there is an isomorphism of groups 
 $A(x_{\ol{\K}}) \cong A(x_{\ol{\F}})$.
 We will extend this result to arbitrary split reductive groups $G$ with geometrically standard fiber $G_\F$. 
 \begin{thm}\label{thm:mcninch-sommers-generalization}
 If $x \in \fg_\O$ is a balanced nilpotent section, then there is a group isomorphism $A(x_{\ol{\K}}) \cong A(x_{\ol{\F}})$. 
 \end{thm}
 \begin{rmk}\label{rmk:comp-group-remark}
 This will be proven in \S\ref{ssec:comp-grp-scheme}. 
 \end{rmk}

We will also prove in Theorem~\ref{thm:integral-comp-group}, that it always possible to enlarge $\O$ so that 
 the identity components for the fibers can be lifted to a normal subgroup scheme $(G^x)^\circ \unlhd G^x$, where the quotient
 $G^x/(G^x)^\circ$ is an affine group scheme whose geometric fibers are $A(x_{\ol{\K}})$ and $A(x_{\ol{\F}})$. 

\subsection{Centralizers for the $G\times \Gm$ action}
In \S\ref{sec:graded-action}, we consider the centralizers for a certain action of $G\times \Gm$ on $\fg$. 
For instance, we will show that if $x \in \fg_\O$ is balanced for the $G$ action, then the centralizer
$(G\times \Gm)^x\subseteq G\times \Gm$ is also smooth, and there exists an isomorphism of component groups for the geometric fibers 
(see Theorem~\ref{thm:graded-analogues}). 
An application of this will be given in Proposition~\ref{prop:reductive-lattice}, which can be used to relate the representation theory for
the reductive quotients
of the $\K$ and $\F$ centralizers by considering the representation theory of $(G\times \Gm)^x$. 

\subsection{Additional comments}
The main source of motivation for this project originated from the author's work with P.~Achar and S.~Riche on the \emph{modular Lustig--Vogan} bijection in \cite{ahr:mlv}, where the structure and representation
theory of $G^x$ plays a crucial role.

It should also be mentioned that the smoothness of $G^x$ result has been verified in a number of cases, and by very different arguments, in 
\cite{ahr:mlv} and \cite{booher:thesis}.

\subsection{Acknowledgements}
The author wishes to express his gratitude to P.~Achar, S.~Riche and G.~McNinch for their helpful comments and suggestions. 

\section{The torsion-free subgroup scheme}

Fix a balanced nilpotent section $x \in \fg_{\O}$ and let 
\[
I_{\text{tor}} = \ker(\O[G^x] \rightarrow \K \otimes_{\O} \O[G^x] \cong \K[G^x]). 
\]
denote the ideal consisting of all the $\omega$-torsion elements of $\O[G^x]$, then $V(I_{\text{tor}})$ is a closed torsion-free subscheme of $G^x$. 

\subsection{}
We will begin by establishing some general properties of this subscheme. 

\begin{lem}\label{lem:tensor-injective}
Let $M$ an arbitrary $\O$-module, and let $N \leq M$ be any free finite-rank $\O$-submodule, then the natural morphism
\[
N\otimes_{\O} N \xrightarrow{\upsilon} M \otimes_{\O} M
\]
is injective. 
\end{lem}
\begin{proof}
Since the functor 
$-\otimes_{\O} \K : \O\text{-mod} \rightarrow \K\text{-mod}$ is exact, the induced map 
\[
N_{\K} \longrightarrow M_{\K}
\]
is also injective, so we can regard $N_{\K}$ as a submodule of $M_{\K}$. There is a commutative diagram
\[
\begin{CD}
            (N\otimes_{\O} N) \otimes \K @= N\otimes_{\O} N_{\K} @ =  N \otimes_{\O}(\K \otimes_{\K} N_{\K}) @= N_{\K} \otimes_{\K} N_{\K}\\
             @V\upsilon \otimes 1 VV  @VVV @VVV @VVV\\
         (M\otimes_{\O} M) \otimes \K @= M\otimes_{\O} M_{\K} @ =  M \otimes_{\O}(\K \otimes_{\K} M_{\K}) @= M_{\K} \otimes_{\K} M_{\K}\\
\end{CD}
\]
where the $``="$ symbols denote natural isomorphisms. The rightmost arrow is the canonical injection
\[
N_{\K} \otimes_{\K} N_{\K} \hookrightarrow M_{\K} \otimes_{\K} M_{\K},
\]
so in particular, $\upsilon \otimes 1$ is injective. 

Composing $- \otimes_{\O} \K$ with the forgetful functor to $\O$-mod, and letting $S = \ker(\upsilon)$, gives the commutative diagram 
\[
\begin{CD}
0      @>>>      S    @>>>    N \otimes_{\O} N @>\upsilon>> M \otimes_{\O} M \\
@.                  @VVV           @VVV                                           @VVV\\ 
0      @>>>      S\otimes_{\O} \K    @>>>    (N \otimes_{\O} N)\otimes_{\O}\K @>\upsilon \otimes 1>> (M \otimes_{\O} M)\otimes_{\O}\K.
\end{CD}
\]
The exactness of the top row implies exactness of the bottom row since $-\otimes_{\O}\K$ is exact. However, the injectivity of $\upsilon \otimes 1$ 
 implies $S \otimes_{\O} \K=0$. Therefore, $S \subseteq N \otimes_{\O} N$ is a torsion submodule of the free finite-rank $\O$-module $N\otimes_{\O} N$,
 and hence, $S =0$. 
\end{proof}

We will also require the following lemma. 
\begin{lem}\label{lem:torsion-free subgroup}
Let $H$ be any affine algebraic group scheme over $\Spec{\O}$, then the closed subscheme $V(I_{tor}) \subseteq H$ is actually a (torsion-free) subgroup scheme, which we will denote by $H_{\text{tf}}$. 
\end{lem}
\begin{proof}
By \cite[I.2.4(6)]{jantzen} it suffices to show 
\begin{equation}\label{eqn:coideal}
\Delta(I_{\text{tor}}) \subseteq I_{\text{tor}}\otimes \O[H] + \O[H] \otimes I_{\text{tor}},
\end{equation}
$\varepsilon(I_{\text{tor}}) = 0$ and $\sigma(I_{\text{tor}}) \subseteq I_{\text{tor}}$ where $\Delta$, $\varepsilon$ and $\sigma$ are the 
comultiplication, counit and antipode for $\O[H]$ respectively. 

Clearly  $I_{\text{tor}} \subseteq I_1$, since $\O[H]/I_{1} \cong \O$ is torsion-free, where $I_1 = \ker(\varepsilon)$. 
Also, since $\sigma: \O[H] \rightarrow \O[H]$ is an $\O$-linear map, it must preserve torsion. So the second and third identities are 
verified and we are left to verify \eqref{eqn:coideal}. 

Let $f \in I_{\text{tor}}$ be arbitrary, then $\Delta(f)$ is also torsion since $\Delta$ is $\O$-linear. 
Now suppose, 
\[
\Delta(f) = \sum_{i=1}^m f_i \otimes h_i,
\] 
and let $M \leq \O[H]$ be the $\O$-submodule generated by $f_i, h_i$, $1\leq i \leq m$. Since $M$ is finitely-generated and 
$\O$ is a DVR, then there exists a decomposition $M = M_{\text{tor}} \oplus M_{\text{free}}$. Let $\{e_1,\dots, e_r\}$ be a basis for $M_{\text{free}}$, then 
for $i=1,\dots, m$,
\[
\begin{aligned}
f_i &= \left( \sum_{j=1}^m a_{i,j}e_j \right) + \varphi_i, \quad
h_i &= \left( \sum_{j=1}^m b_{i,j}e_j \right) + \psi_i,
\end{aligned}
\]
where $a_{i,j}, b_{i,j} \in \O$ and $\varphi_i,\, \psi_i \in M_{\text{tor}}$. Let $\alpha_i = \sum_{j=1}^m a_{i,j}e_j $ and 
$\beta_i = \sum_{j=1}^m b_{i,j}e_j$, so that 
\[
f_i \otimes h_i = \alpha_i \otimes \beta_i + (\alpha_i \otimes \psi_i + \varphi_i \otimes \beta_i + \varphi_i \otimes \psi_i).
\]
Clearly, ${f_i \otimes h_i - \alpha_i \otimes \beta_i \in I_{\text{tor}}\otimes \O[G^x] + \O[G^x] \otimes I_{\text{tor}}}$, and is, in particular, torsion. 

On the other hand, the $\alpha_i \otimes \beta_i$ terms are in the image of the natural map 
\[
 M_{\text{free}} \otimes M_{\text{free}} \longrightarrow \O[H]\otimes \O[H]. 
\]
By Lemma~\ref{lem:tensor-injective}, this map is injective, and thus $M_{\text{free}} \otimes M_{\text{free}}$ can be regarded as a free finite-rank $\O$-submodule
of $\O[H]\otimes \O[H]$. It follows that
\[
\sum_{i=1}^m \alpha_i \otimes \beta_i = \Delta(f) - \sum_{i=1}^m (\alpha_i \otimes \psi_i + \varphi_i \otimes \beta_i + \varphi_i \otimes \psi_i) 
\]
is both torsion and lies in the free submodule $M_{\text{free}} \otimes M_{\text{free}}$, and therefore must be zero. 
\end{proof}
\begin{rmk}
The property that $V(I_{\text{tor}})$ is a subgroup scheme, also appears to follow from a more general property,
 stated at the beginning of \S 3.1 in\cite{gille-mb}.
\end{rmk}


\subsection{}\label{subsec:connected-case}
In this subsection, it will be proven that $G^x_{\text{tf}}$ is smooth, and that $G^x_{\text{tf}, \ol{\F}}$ 
contains the identity component $(G^x_{\ol{\F}})^{\circ}$ (see \cite[Definition~13.12]{mil:ALA} for the definition of the identity component over a general field). 
By Remark~\ref{rmk:smooth-fiber}, 
the identity component of $G^x_{\ol{F}}$ is also reduced. 
Our strategy will be to compare the distribution algebras of the $\K$ and $\F$ centralizers. 

Let us first recall that the $\K$ and $\F$ centralizers admit the Levi decompositions
\begin{equation}\label{eqn:levi}
G^{x}_{\K} = G^{x}_{\K, {\text{red}}} \ltimes  G^{x}_{\K, {\text{unip}}}, \quad G^{x}_{\F} = G^{x}_{\F, {\text{red}}} \ltimes  G^{x}_{\F, {\text{unip}}}
\end{equation}
(cf. \cite[Corollary~29]{mcninch2004}). 
By \cite[Corollary 9.2.2]{mcninch2016}, the $\K$ and $\F$ reductive quotients also have the same root datum. 
Moreover, by \cite[Theorem 28]{mcninch2004}, the unipotent radicals $G^{x}_{\K, {\text{unip}}}$ and $G^{x}_{\F, {\text{unip}}}$ are both
\emph{split}, and have the same rank (i.e. as schemes, they are both isomorphic to affine spaces of the same dimension). Thus, there is an isomorphism of 
graded vector spaces
\begin{equation}\label{eqn:unipotent-distributions}
 \Dist(G^{x}_{\K, {\text{unip}}}) = \Dist(\mathbb{A}_\K^d, 0), \quad  \Dist(G^{x}_{\F, {\text{unip}}}) = \Dist(\mathbb{A}_\F^d, 0)
\end{equation}
where $d = \dim\, G^{x}_{\K, {\text{unip}}} = \dim\, G^{x}_{\F, {\text{unip}}}$. These distributions are explicitly calculated in \cite[I.7.3]{jantzen}.

\begin{lem}\label{lem:dist-compare}
For all $n\geq 0$, 
\[
\dim\, \Dist_n(G^x_{\K}) = \dim\, \Dist_n(G^x_{\F}). 
\]
\end{lem}
\begin{proof}
Apply \eqref{eqn:levi}, and observe that since $G^x_{\text{red}, \K}$ and $G^x_{\text{red}, \F}$ have the same root-datum, 
then the classification results from  \cite[II.1]{jantzen} imply
the existence of a Kostant $\O$-form $\uenv_{\O} \subseteq \Dist(G^x_{\text{red}, \K})$ which satisfies
$\F \otimes_{\O} \uenv_{\O} \cong \Dist(G^x_{\text{red},\F})$. So in particular, 
\[
\dim\, \Dist_n(G^x_{\text{red}, \K}) = \dim\, \Dist_n(G^x_{\text{red}, \F})
\]
for all $n\geq 0$.  

The result now follows by  tensoring both sides with the respective unipotent distribution algebras and applying \eqref{eqn:unipotent-distributions}. 
\end{proof}


\begin{defn}
By \cite[I.7.4]{jantzen}, an affine scheme $X$ over $\O$ is said to be \emph{infinitesimally flat} at
$z \in X(\O)$ provided $\O[X]/I_{z}^{n+1}$ are flat $\O$-modules for all $n\geq 0$, where $I_{z}$ is the
ideal corresponding to $z$. 
\end{defn}

The infinitesimal flatness property is necessary in order for distribution algebras to work ``nicely'' for an $\O$-group scheme.
\begin{lem}\label{lem:inf-flat}
Both $G^x$ and $G^x_{\text{tf}}$ are infinitesimally flat at $1 \in G^x_{\text{tf}}(\O) = G^x(\O)$, and
$I_{tor} \subseteq \bigcap_{n\geq 1} I_1^n$. 							
\end{lem}
\begin{proof}
Let $I_{1,\K}$ and $I_{1,\F}$ be the augmentation ideals for the corresponding fibers. From the map
\[
\O[G^x] \rightarrow \F\otimes_{\O}\O[G^x] \cong \O[G^x]/\omega\O[G^x],
\]
we can see that $I_{1,\F} = \frac{I_1 + \omega\O[G^x]}{\omega\O[G^x]}$, and more generally,
\[
I_{1,\F}^n = \frac{I_1^n + \omega\O[G^x]}{\omega\O[G^x]}.
\]
This gives 
\[
  \F[G^x]/I_{1,\F}^n \cong \frac{\O[G^x]}{I_{1}^n + \omega\O[G^x]}.
\]
On the other hand, 
\[
\F \otimes_{\O} (\O[G^x]/I_1^n) \cong \frac{\O[G^x]/I_{1}^n}{\omega(\O[G^x]/I_{1}^n)}
   \cong \frac{\O[G^x]/I_{1}^n}{(\omega\O[G^x] + I_{1}^n)/I_{1}^n} \cong \F[G^x]/I_{1,\F}^n. 
\]
So there exist morphisms
\[
\K[G^x]/I_{1,\K}^n \cong \K \otimes_{\O}(\O[G^x]/I_1^n) \leftarrow \O[G^x]/I_1^n 
	\rightarrow \F\otimes_{\O}(\O[G^x]/I_1^n) \cong \F[G^x]/I_{1,\F}^n.
\]
By Lemma~\ref{lem:dist-compare}, the dimensions of the left and right side are equal, so that $\O[G^x]/I_1^n$ are torsion-free, and hence flat, 
for all $n$. Therefore, $G^x$ i infinitesimally flat at $1$. 
Finally,  since ${I_{\text{tor}} \subseteq \bigcap_{n\geq 1} I_1^n}$, it also follows that  $G^x_{\text{tf}}$ is infinitesimally flat at $1$. 
\end{proof}

The preceding lemma also implies that $\Dist_n(G^x_{\text{tf}}) \cong \Dist_n(G^x)$ for all $n\geq 0$. 
The infinitesimal flatness for both groups allows us to apply \cite[I.7.4(1)]{jantzen}, which gives
\begin{equation}\label{eqn:equal-dist_n}
 \Dist_n(G^x_{\text{tf},\F}) \cong \F \otimes_{\O} \Dist_n(G^x_{\text{tf}}) \cong  
     \F \otimes_{\O} \Dist_n(G^x) \cong \Dist_n(G^x_{\F}),
 \end{equation}
 for all $n \geq 0$. In particular, $\Dist(G^x_{\text{tf}, \F}) \cong \Dist(G^x_\F)$ as filtered algebras.

 
  We may now prove the main result of this section. 
 \begin{prop}\label{prop:connected}
 The group scheme $G^x_{\text{tf}}$ is smooth and ${(G^x_{\ol{\F}})^{\circ} \subseteq G^x_{\text{tf},\ol{\F}}}$. 
 \end{prop}
 \begin{proof}
 By construction, $G^x_{\text{tf}, \F} \subseteq G^x_{\F}$, and thus by  \cite[I.7.17(7)]{jantzen},
 $\Dist(G^x_{\text{tf},\F}) \subseteq \Dist(G^x_{\F})$ as filtered algebras. So 
 $\Dist_n(G^x_{\text{tf},\F}) \subseteq \Dist_n(G^x_{\F})$ for all $n\geq 0$. 
  However, by \eqref{eqn:equal-dist_n}, it follows that $\dim_\F \Dist_n(G^x_{\text{tf},\F})= \dim_\F \Dist_n(G^x_{\F})$,
 and hence, $\Dist_n(G^x_{\text{tf},\F}) = \Dist_n(G^x_{\F})$ for all $n \geq 0$. Therefore, 
 \begin{equation}\label{eqn:equal-dist}
 \Dist(G^x_{\text{tf},\F}) = \Dist(G^x_\F).
 \end{equation}
 
 Let us now base change to $\ol{\F}$, where we note that \eqref{eqn:equal-dist} also holds over $\ol{\F}$ by \cite[I.7.4(1)]{jantzen}. 
 The definition of the identity component implies $1_{G^x_{\ol{\F}}} \in (G^x_{\ol{\F}})^\circ(\ol{\F})$, and so it follows that 
 $
\Dist((G^x_{\ol{\F}})^\circ) = \Dist(G^x_{\text{tf},\ol{\F}}).
$
Now since $(G^x_{\ol{\F}})^\circ$ is an irreducible subgroup scheme of $G^x_{\ol{\F}}$, then by 
\eqref{eqn:equal-dist} and \cite[I.7.17(7)]{jantzen}, 
$
(G^x_{\ol{\F}})^\circ \subseteq G^x_{\text{tf}, \ol{\F}},
$
proving the second statement of the proposition.

To prove the first statement, we first recall that by \cite[I.7.1(2)]{jantzen}, 
 \[
 \dim_\F \Dist^+_n(G^x_{\text{tf},\ol{\F}})= \dim_\F \Dist^+_n(G^x_{\ol{\F}})
 \]
 for all $n \geq 1$. Thus, by \cite[I.7.7]{jantzen}, the Lie algebras of $G^x_{\text{tf}, \ol{\F}}$ and $G^x_{\ol{\F}}$ have the same dimension. 
 On the other hand, 
 \[
 (G^x_{\ol{\F}})^\circ \subseteq G^x_{\text{tf}, \ol{\F}} \subseteq G^x_{\ol{\F}}
 \]
 implies $\dim G^x_{\text{tf},\ol{\F}} = \dim G^x_{\ol{\F}}$, and hence, $G^x_{\text{tf}, \ol{\F}}$ is smooth by \cite[I.7.17(1)]{jantzen}.
 Likewise, $G^x_{\text{tf}, \ol{\K}}$ is automatically smooth by \cite[I.7.17(2)]{jantzen}. Thus, by Lemma~\ref{defn:smooth}, $G^x_{\text{tf}}$ 
 will be smooth provided $\O[G^x_{\text{tf}}]$ is torsion-free. However, this property holds from the definition of $G^x_{\text{tf}}$, and so we are done. 
 \end{proof}

\section{Smoothness and component groups of centralizers}



To simplify our arguments, we will assume throughout this section that $\O$ is such that $\F=\ol{\F}$ is algebraically closed, 
unless specified otherwise. 

 \subsection{Diagonalizable group schemes}\label{ssec:diag-grp-schms}
 For any commutative group $\Lambda$ and a commutative ring $\bk$, the \emph{diagonalizable group scheme} over $\bk$ associated to $\Lambda$, denoted
  $\Diag(\Lambda)$, is defined by setting $\bk[\Diag(\Lambda)] := \bk[\Lambda]$, where $\bk[\Lambda]$ is the group algebra for $\Lambda$ with
  the usual Hopf algebra structure (cf. \cite[I.2.5]{jantzen}). 
\begin{lem}\label{lem:diag-pullback}
Let $\bk$ be a commutative integral ring, and suppose 
$X= \Diag(\Lambda_X)$, $Y = \Diag(\Lambda_Y)$ and $Z = \Diag(\Lambda_Z)$ are diagonalizable group schemes over $\bk$ with 
morphisms $X \xrightarrow{\varphi} Z$ and $Y \xrightarrow{\psi} Z$, then  
\[
X \times_Z Y \cong \Diag(\Lambda),
\]
where $\Lambda$ is a pushout induced by two uniquely determined group homomorphisms $\Lambda_Z \xrightarrow{\varphi'} \Lambda_X$ and 
$\Lambda_Z \xrightarrow{\psi'} \Lambda_Y$. 
\end{lem}
\begin{proof}
Since $\bk$ is integral, then the isomorphism  \cite[I.2.5(2)]{jantzen} ensures that the morphisms $\varphi$ and $\psi$ are induced by 
unique group homomorphisms 
$\Lambda_Z \xrightarrow{\varphi'} \Lambda_X$ and 
$\Lambda_Z \xrightarrow{\psi'} \Lambda_Y$. 
More precisely, if we identify $\bk[X] = \bk[\Lambda_X]$, $\bk[Y] = \bk[\Lambda_Y]$ and $\bk[Z] = \bk[\Lambda_Z]$, where the Hopf algebras on the right 
are the group algebras for $\Lambda_X$, $\Lambda_Y$ and $\Lambda_Z$ respectively, then the comorphisms $\varphi^*$ and $\psi^*$ are 
the group algebra homomorphisms induced by $\varphi'$ and $\psi'$ respectively. 

Now let $\Lambda$ be the pushout of $\varphi'$ and $\psi'$, then 
\begin{equation}\label{eqn:group-pushout}
\Lambda = \frac{\Lambda_X \times \Lambda_Y}{\la (\varphi'(\lambda_Z),1)(1,\psi'(\lambda_Z)^{-1}) \, \mid \, \lambda_Z \in \Lambda_Z  \ra},
\end{equation}
where we use multiplicative notation to denote the group structure. 
Let us  employ the natural identification $\Lambda' \subseteq \bk[\Lambda']$ for an abelian group $\Lambda'$ (i.e. identifying $\Lambda'$ with the
``group-like" elements of $\bk[\Lambda']$), so that by definition, $\Lambda'$ gives a $\bk$-basis of $\bk[\Lambda']$. 
We will also make use of the isomorphism of algebras 
\begin{equation}\label{eqn:tensor-direct-product}
\bk[\Lambda_X] \otimes_\bk \bk[\Lambda_Y] \xrightarrow{\sim} \bk[\Lambda_X \times \Lambda_Y],
\end{equation}
induced by sending $\lambda_X \otimes_\bk \lambda_Y \mapsto (\lambda_X,\lambda_Y)$ for any $\lambda_X \in \Lambda_X$ and $\lambda_Y \in \Lambda_Y$. 

Now let us note that the tensor product $\bk[\Lambda_X]\otimes_{\bk[\Lambda_Z]}\bk[\Lambda_Y]$ coincides with the pushout along $\varphi^*$ and $\psi^*$ in the 
category of commutative $\bk$-algebras. 
By definition $\bk[\Lambda_X]\otimes_{\bk[\Lambda_Z]}\bk[\Lambda_Y] \cong \bk[\Lambda_X] \otimes_\bk \bk[\Lambda_Y]/I$, where $I$ is the $\bk$-submodule
generated by $f\phi^*(h)\otimes g - f\otimes \psi^*(h)g$ for all $f \in \bk[X]$, $g \in \bk[Y]$ and $h \in \bk[Z]$. 
In particular, $I$ is spanned by 
 \[
 \lambda_X\varphi'(\lambda_Z)\otimes \lambda_Y - \lambda_X \otimes \psi'(\lambda_Z)\lambda_Y,
 \]
 for all $\lambda_X$, $\lambda_Y$ and $\lambda_Z$. 

Observe now that
 \begin{align*}
  \lambda_X\varphi'(\lambda_Z)\otimes \lambda_Y = \lambda_X \otimes \psi'(\lambda_Z)\lambda_Y &\iff  \varphi'(\lambda_Z) \otimes \psi'(\lambda_Z)^{-1} = 1\otimes 1 \\
  	&\iff (\varphi'(\lambda_Z),\psi'(\lambda_Z)^{-1}) = 1
 \end{align*}
for $\lambda_X$, $\lambda_Y$ and $\lambda_Z$, where the second ``$\iff$" arises from \eqref{eqn:tensor-direct-product}. From this, we can see that $\bk[\Lambda_X]\otimes_{\bk[\Lambda_Z]}\bk[\Lambda_Y]$ 
 can be equivalently obtained from $\bk[\Lambda_X \times \Lambda_Y]$ by imposing the relation 
 $(\varphi'(\lambda_Z),\psi'(\lambda_Z)^{-1}) = 1$ for all $\lambda_Z \in \Lambda_Z$. Comparing with the description in \eqref{eqn:group-pushout}
gives the group algebra $\bk[\Lambda]$. Therefore, we are done.
\end{proof}

  Let $\bk$ be any commutative ring, then we say that an affine $\bk$-scheme $X$ is \emph{constant}, provided the associated $\bk$-functor 
 \[
 X: \{\bk\text{-algebras}\} \rightarrow \{\text{sets}\}
 \]
 is a constant functor. If the cardinality $|X(\bk)| = r < \infty$, then there exists an algebra isomorphism
  $\bk[X] \cong \bk^{\times r}$, where the right hand side is an $r$-fold direct product with pointwise 
 multiplication. 
 An affine group scheme is called constant if it is constant as a scheme. 
  
 \begin{lem}\label{lem:diag-constant}
 Let $\Lambda$ be a finite abelian group with $r = |\Lambda|$. If $p \nmid r$, then $\Diag(\Lambda)$ is a constant $\O$-group scheme 
 (i.e. $\O[\Lambda] \cong \O^{\times r}$ as an $\O$-algebra). 
 \end{lem}
 \begin{proof}
 Since $\Lambda$ is a finite abelian group, then 
 \[
 \Lambda \cong  \Z/n_1\Z \times  \Z/n_2\Z \times \cdots \times  \Z/n_t\Z,
 \]
 where the $n_1,\dots, n_t \in \Z$ may have repeated multiplicities. In particular, 
 \[
 \O[\Lambda] \cong \O[\Z/n_1\Z] \otimes  \O[\Z/n_2\Z] \otimes \cdots \otimes  \O[\Z/n_t\Z].
 \]
 Thus, $\Diag(\Lambda)$ is constant if $\Diag(\Z/n_i\Z)$ is constant for all $i$. The condition
 $p \nmid r$ clearly implies $p\nmid n_i$ for all $i$. Therefore, without loss of generality, we may assume 
 $\Lambda$ is cyclic. 
 
 Supposing now that $\Lambda = \Z/r\Z$, gives 
 \[
 \O[\Lambda] = \O[z]/\la z^r-1\ra.
 \]
 Our assumption on $\O$ at the beginning of this section, implies that $\O$ (and hence $\K$) contain all roots of unity which are co-prime to $p$. 
 Thus,
 \[
 z^r-1 = (z-\zeta_1)(z-\zeta_2)\cdots (z-\zeta_r),
 \]
 where $\zeta_1,\dots,\zeta_r \in \O^*\subset \K^*$ are the primitive $r$-th roots of unity. 
 We now define a ring homomorphism 
 \[
 \begin{aligned}
 \varphi: \O&[\Lambda] \rightarrow \O^{\times r} \\
                f &\mapsto (f(\zeta_1),f(\zeta_2),\dots, f(\zeta_r)). 
 \end{aligned}
 \]
 The result will follow if we can prove that $\varphi$ is an isomorphism. 
 
 It suffices to prove that $\varphi$ is surjective as an $\O$-module homomorphism. For $i=1,\dots, r$, set 
 $f_i = \prod_{j\neq i}(z-\zeta_j)$. Then $f(\zeta_j) = 0$ for all $j\neq i$ and 
 \[
 f_i(\zeta_i) = \prod_{j\neq i}(\zeta_i-\zeta_j).
 \]
Our assumption $p \nmid r$ implies that the reductions $\ol{\zeta_1},\ol{\zeta_2},\dots \ol{\zeta_r} \in \F = \O/\la \omega\ra$ are 
 all distinct (where $\omega \in \O$ is the uniformizer). In particular, for $i\neq j$, $\ol{\zeta_i} - \ol{\zeta_j} \neq 0$, and hence, 
 $\omega \nmid (\zeta_i - \zeta_j)$ so that $(\zeta_i - \zeta_j)$ is invertible in $\O$. Thus, $f_i(\zeta_i) \in \O$ is also invertible for all $i$. 
 
 Finally, if $\epsilon_i \in \O^{\times r}$ denotes the $i$-th coordinate function for $i=1,\dots, r$, then 
 $
 \varphi(f_i) = f_i(\zeta_i)\epsilon_i
 $
 for all $i$. The invertibility of the $f_i(\zeta_i)$ imply that the $ f_i(\zeta_i)\epsilon_i$ form a basis for the free $\O$-module $\O^{\times r}$, and therefore, 
 $\varphi$ is surjective. 
 \end{proof}
 




Following the convention from \cite[p.~5]{mcninch2016}, we say that a subgroup scheme $S \subseteq H$ of an $\O$-group scheme 
$H$ is a \emph{maximal torus} if the subgroup schemes $S_\K \subseteq H_\K$ and $S_\F \subseteq H_\F$ are both maximal tori. 

\begin{lem}\label{lem:torus-intersect}
Let $C_{\F}$ be a nilpotent orbit, then there exists a balanced nilpotent section $x \in \fg_\O$ such that 
$x_{\F} \in C_{\F}$ and 
$S_\F  = (T_\F \cap G^x_\F)^\circ$ is a maximal torus for $G^x_\F$.
\end{lem}
\begin{proof}
By \cite[Theorem~1.2.1(a)]{mcninch2016}, any $x_\F \in C_{\F}$ lifts to some 
balanced nilpotent section $x \in \fg_\O$. Suppose now that $y_\F \in C_\F$ is arbitrary, and let $y \in \fg_\O$,
be a corresponding balanced nilpotent section which lifts $y_\F$.  If $S'_\F \subseteq G^{y_\F}_{\F}$ is a maximal torus for the centralizer,
then there must exist a maximal torus $T'_\F$ of $G_\F$ such that 
$S'_\F \subseteq T'_\F$. 
Now since the maximal tori for $G_\F$ are conjugate, then there exists $g \in G_\F(\F)$ such that 
$T_\F = gT'_\F g^{-1}$. 
 Let $x_\F = g^{-1}y_{\F} g$ and fix a balanced lift $x \in \fg_\O$,
 then $x_\F \in C_{\F}$ and 
$S_\F = g^{-1} S'_\F g \subseteq T_\F$,  must also be a maximal torus in $G^{x_\F}_\F$.

It will follow that $S_\F = (T_\F \cap G^x_\F)^\circ$ provided $T_\F\cap G^x_\F$ is reduced. To see why this is true, let us first fix 
the root space decomposition
\[
\fg_\F = \ft_\F \oplus \bigoplus_{\alpha \in \Phi \subset \bX} (\fg_\F)_\alpha,
\]
so that the structure of $\fg_\F$ as a $T_\F$-module is given by the comodule map 
\begin{align*}
\Delta_{\fg_\F}: \fg_\F &\rightarrow \F[T] \otimes \fg_\F\\
v &\mapsto \sum_{\mu \in \Phi\cup \{0\}} \mu(t)\otimes v_\mu,
\end{align*}
where $v_\mu$ denotes the projection onto the $\mu$ weight space of $\fg_\F$. 
The centralizer of $x_\F$ in $T_\F$ is then determined by the abelian group 
\[
\bX^{x_\F} = \bX/\la \mu \in \bX \, \mid \,  (x_\F)_\mu \neq 0\ra,
\]
(i.e. $T\cap G^x_\F \cong \Diag(\bX^{x_\F})$). From the properties of diagonalizable group schemes, this is reduced if and only if 
$\bX^{x_\F}$ contains no $p$-torsion. However, the latter property follows immediately from \cite[Definition 2.11 and Theorem 5.2]{her:scrg}. 
\end{proof}

The following proposition will enable us to relate the torus characters between $G^x_\K$ and $G^x_\F$ representations. 
\begin{prop}\label{prop:lift-torus}
Let $x \in \fg_\O$ and $S_\F \subseteq G^x_\F$ be as in Lemma~\ref{lem:torus-intersect}, then $S_\F$ lifts to a 
split torus $S \subseteq G^x$ such that $S_\K \subseteq G^x_\K$ is a maximal torus which is split. 
\end{prop}
\begin{proof}
First observe that since  $S_\F \subseteq G^x_\F$ is irreducible, and $ \Dist(G^x_\F) =\Dist(G^x_{\text{tf},\F})$ by \eqref{eqn:equal-dist}, then
from \cite[I.7.17(7)]{jantzen}, it follows that $S_\F \subseteq G^x_{\text{tf},\F}$. The fact that
$G^x_{\text{tf}}$ is smooth by Proposition~\ref{prop:connected} now allows us to apply \cite[Theorem~2.1.1(c)]{mcninch2016} 
to lift the embedding $\varphi: \Gm^r \hookrightarrow G^x_{\text{tf}, \F}$ (corresponding to the inclusion 
$S_\F \subseteq G^x_{\text{tf}, \F}$), to a map $\psi: \Gm^r \rightarrow G^x_{\text{tf}}\subseteq G^x$. Moreover, by the
same reasoning as in the proof of \cite[Lemma~4.1]{ahr:rtdrg}, it can be verified that $\psi$ is actually a closed embedding. 

Thus, if we set $S = \psi(\Gm^r) \subseteq G^x_{\text{tf}}$, where $S \cong \Gm^r$,
then  $S_\K = \psi_\K(\Gm^r) \subseteq G^x_{\K}$ is a split rank $r$-torus contained in $G^x_\K$. 
Finally, since the maximal tori for $G^x_\K$ and $G^x_\F$ have the same rank by \cite[Corollary 9.2.2]{mcninch2016}, then $S_\K$ must also be maximal. 
\end{proof}
\begin{rmk}\label{rmk:geo-fibers}
In particular, $S_{\ol{\K}} \subseteq (G^x_{\ol{\K}})^\circ$ and 
$S_{\F} \subseteq (G^x_{\F})^\circ$ are maximal tori for the connected components of the geometric fibers. 
\end{rmk}

In order to reduce our component group calculations to the case of diagonalizable group schemes, we will need the following proposition. 
\begin{prop}\label{prop:torus-intersect}
Let $Z = Z(G)$ denote the center of $G$, let $Z'\subseteq Z$ be any diagonalizable subgroup scheme, and let $S \subseteq G$ be any 
split $\O$-torus, then $S\cap Z' \subseteq G$ is diagonalizable. 
\end{prop}
\begin{proof}
First recall that the center $Z$ is diagonalizable (see \cite[II.1.6, II.1.8]{jantzen}). The strategy of the proof will be to first construct a split $\O$-torus $T' \subseteq G$ with $S\subseteq T'$ and $Z'\subseteq T'$, this will allow the intersection of 
$S\cap Z'$ to be taken \emph{inside} of $T'$. The result will then follow from Lemma~\ref{lem:diag-pullback}. 

To construct $T'$, first set $H = C_G(S)$, where $C_G(S)$ denotes the centralizer of $S$ in $G$. By \cite[Proposition~2.2.1]{mcninch2016}, $H$ is a smooth reductive
group scheme over $\O$ with connected fibers. 
Now let 
$T'_{\F} \subseteq H_\F$ be a maximal split torus (of rank-$r$) for $H$ (recall that $\F = \ol{\F}$ by our assumption, so such a torus exists). As in the proof of 
Proposition~\ref{prop:lift-torus}, $T'_{\F}$ can be lifted to a rank-$r$ split torus of $H$. Let $T'$ denote this lift. Thus, $T'_\K \subseteq H_\K$ is a rank-$r$ split 
torus of $H_\K$. By \cite[Corollary~2.1.2]{mcninch2016}, the maximal tori for $H_\K$ and $H_\F$ must have the same dimension, and hence the same rank if they
are split. Therefore, $T'_\K$ is maximal in $H_\K$. 

Let $Z(H) \subseteq H$ denote the center of $H$. By definition, $S\subseteq Z(H)$ and $Z \subseteq Z(H)$ (and hence $Z' \subseteq Z(H)$). 
For $\bk \in \{\K, \F\}$, 
\begin{equation}\label{eqn:k-containment}
Z(H)_\bk =Z(H_\bk) \subseteq T'_\bk,
\end{equation}
where the containment on the right holds because  $T'_\bk \subseteq H_\bk$ is a maximal torus over a field and $Z(H)_\bk$ is the center. 
So in particular, $S_\bk \subseteq T'_\bk$. 

To prove $S \subseteq T'$, first consider the inclusion
\[
\varphi: S\cap T' \hookrightarrow S,
\]
and observe that $S \subseteq T'$ if and only if $\varphi$ is an isomorphism. 
Since base-change commutes with taking fiber products, and thus commutes with taking intersections, the base-changes of $\varphi$ to $\bk$,
\[
\varphi_\bk: (S \cap T')_\bk \rightarrow S_\bk,
\]
coincide with the inclusion $S_\bk \cap T'_\bk \hookrightarrow S_\bk$. Thus, by \eqref{eqn:k-containment}, $\varphi_\bk$ is an isomorphism (i.e. is surjective). 
On the level of algebras, $\O[S\cap T'] = \O[S] \otimes_{\O[H]}\O[T']$ where the comorphism 
$\varphi^*$ is surjective, so it suffices to show that $\varphi_*$ is injective (as an $\O$-module morphism). 
Base changing to $\K$ induces a commutative diagram
\[
\begin{CD}
     \O[S]    @>\varphi^*>>   \O[S]\otimes_{\O[H]}\O[T']   \\
         @V \psi VV           @VV \psi' V                                           \\ 
     \K[S]    @>\varphi^*_\K>\sim>    \K[S]\otimes_{\K[H]}\K[T'].
\end{CD}
\]
The diagonalizability of $S$ implies that $\O[S]$ is torsion-free, and hence $\psi = 1 \otimes \id$ must be injective. Thus 
$\varphi^*_\K \circ \psi$ is injective, and by commutativity, $\psi'\circ \varphi^*$ is injective. It then follows that
$\varphi^*$ must be injective, and is therefore an isomorphism. 

By a similar argument, we can show that the map $Z' \cap T' \hookrightarrow Z'$ is an isomorphism (using the fact that $Z'$ is diagonalizable), and thus 
$Z' \subseteq T'$. 

We have established that $S \subseteq T'$ and $Z' \subseteq T'$. This allows us to identify 
\[
S\cap Z' = S \times_{T'} Z'.
\]
Thus, it follows from Lemma~\ref{lem:diag-pullback} that $S\cap Z'$ is diagonalizable. 
\end{proof}

 \subsection{Component groups of the geometric fibers}\label{subsec:comp-fibers}
We will now prove that the component groups $A(x_{\ol{\K}})$ and $A(x_\F)$ of the geometric fibers, have the same cardinality. 

\begin{lem}\label{lem:simple-comp-group}
If $G$ is a simple group (with $G_\F$ geometrically standard), then $|A(x_{\ol{\K}})|=|A(x_{\F})|$. 
\end{lem}
\begin{proof}
Let $\bk \in \{\K,\F\}$ and denote $H = G^x$ for a balanced nilpotent section $x \in \fg_\O$ and $S \subseteq H$ as in Proposition~\ref{prop:lift-torus}, so that 
$S_{\ol{\bk}}$ is a maximal torus for $H^\circ_{\ol{\bk}}$ by Remark~\ref{rmk:geo-fibers}. 

Let $G' = G_{\text{ad}}$ and $H' = G_{\text{ad}}^x$, then the arguments in the proof of \cite[Lemma~2.33]{ls} can be used to show 
that there exists an isogeny
\[
1\rightarrow Z_{\ol{\bk}} \rightarrow H_{\ol{\bk}} \rightarrow H'_{\ol{\bk}} \rightarrow 1,
\]
obtained by restricting the ``covering space" isogeny
\[
1\rightarrow Z_{\ol{\bk}} \rightarrow G_{\ol{\bk}} \rightarrow G'_{\ol{\bk}} \rightarrow 1,
\]
down to $H_{\ol{\bk}}$. 
(In particular, $Z_{\ol{\bk}} \subseteq G_{\ol{\bk}}$ is the center of $G_{\ol{\bk}}$.)

Now since $H^\circ_{\ol{\bk}}$ gets 
mapped onto $H'^{\circ}_{\ol{\bk}}$, there exists a sequence
\begin{equation}\label{eqn:f-group-sequence}
1 \rightarrow \frac{Z_{\ol{\bk}}H^\circ_{\ol{\bk}}}{H^\circ_{\ol{\bk}}} \rightarrow A(x_{\ol{\bk}}) \rightarrow A'(x_{\ol{\bk}}) \rightarrow 1,
\end{equation}
where $\displaystyle A'(x_{\ol{\bk}}) = \frac{H'_{\ol{\bk}}}{H'^{\circ}_{\ol{\bk}}}$. We can also identify 
$\displaystyle \frac{Z_{\ol{\bk}}H^\circ_{\ol{\bk}}}{H^\circ_{\ol{\bk}}} \cong \frac{Z_{\ol{\bk}}}{Z_{\ol{\bk}}\cap H^\circ_{\ol{\bk}}}$. 
By Remark~\ref{rmk:bala-carter} and \cite[Theorem 36]{mcninch-sommers}, $|A'(x_{\ol{\K}})| = |A'(x_{\F})|$. Thus, by \eqref{eqn:f-group-sequence}, it suffices to show 
$|Z_{\ol{\K}}\cap H^\circ_{\ol{\K}}| = |Z_{\F}\cap H^\circ_{\F}|$. 

However, since $S_{\ol{\bk}}$ is a maximal torus, and $Z_{\ol{\bk}}\cap H^\circ_{\ol{\bk}}$ is central, then it follows that
\[
 Z_{\ol{\bk}}\cap H^\circ_{\ol{\bk}} \subseteq S_{\ol{\bk}}.
\]
This gives, $Z_{\ol{\bk}}\cap H^\circ_{\ol{\bk}} =  Z_{\ol{\bk}}\cap S_{\ol{\bk}}$. Observe that the intersection on the right hand side actually 
arises from an intersection of $\O$-group schemes since 
$Z_{\ol{\bk}}$ and $S_{\ol{\bk}}$ are the (geometric) fibers of the subgroup schemes $Z \subseteq G$ and $S \subseteq G$. In other words,
\[
(Z\cap S)_{\ol{\bk}} =  Z_{\ol{\bk}}\cap S_{\ol{\bk}}.
\]

By Proposition~\ref{prop:torus-intersect}, $Z\cap S$ is a diagonalizable subgroup scheme of $Z$. 
Now observe $Z\cong \Diag(\Lambda)$ with $\Lambda = \bX/\Z\Phi$. The geometrically standard assumption implies that 
$p \nmid |\bX/\Z\Phi|$. In particular, identifying $Z\cap S = \Diag(\Lambda')$ gives $p\nmid |\Lambda'|$ since $\Lambda'$ is a quotient of $\Lambda$. 
 Thus, Lemma~\ref{lem:diag-constant} now implies that $Z\cap S$ constant, and hence 
$|(Z\cap S)_{\ol{\K}}| = |(Z\cap S)_{\F}|$. Therefore, we are done. 
\end{proof}

Now we consider the case of semisimple groups. 

\begin{lem}\label{lem:semisimple-comp-group}
Let $G$ be a semisimple group (with $G_\F$ geometrically standard), then $|A(x_{\ol{\K}})|=|A(x_{\F})|$. 
\end{lem}
\begin{proof}
In this proof $\bk$ will denote either $\K$ or $\F$. 

First suppose that 
\[
G = G_1 \times \cdots \times G_r,
\]
where the $G_i$ are simple (and hence geometrically standard over $\F$ since $G$ is geometrically standard of $\F$ by \cite[\S 3.1~(S2)]{mcninch2016}). 
Now note that 
since the factors of $G$ act independently
 on the corresponding factors of $\fg_\O$, then
for \emph{any} nilpotent section 
\[
x = (x_1,\dots,x_r) \in (\fg_1)_\O \times \cdots \times (\fg_r)_\O,
\]
the centralizer $G^x$ satisfies
\[
G^x = G_1^{x_1}\times \cdots \times G_r^{x_r},
\]
where $G_i^{x_r}\subseteq G_i$ is the centralizer of $x_i$.
 To simplify notation, set $H_i = G_i^x$ for all $i$, and set $H = G^x$. 
Thus,
 \[
 H_{\bk} = (H_1)_{\bk} \times \cdots \times (H_r)_{\bk},
 \]
 and hence, $H_{\bk}$ is smooth if and only if $(H_i)_{\bk}$ is smooth for all $i$. This implies that 
 $x \in \fg_\O$ is a \emph{balanced} for $G$ if and only if for all $i$,  $x_i$ is balanced for $G_i$.
Thus, if we assume that $x$ is balanced, then from the identity
$H_{\ol{\bk}}^\circ = (H_1)^\circ_{\ol{\bk}} \times \cdots \times (H_r)^\circ_{\ol{\bk}}$,
it follows that
\[
A(x_{\ol{\bk}}) =  A((x_1)_{\ol{\bk}}) \times \cdots \times A((x_r)_{\ol{\bk}}).
\]
Therefore, $|A(x_{\ol{\K}})| = |A(x_{\F})|$ by Lemma~\ref{lem:simple-comp-group}.


More generally, if $G$ is an arbitrary semisimple group, then there exists an isogeny 
\begin{equation}\label{eqn:semi-simple-isogeny}
 1 \rightarrow Z' \rightarrow \prod_{i=1}^r G_i \xrightarrow{\pi} G \rightarrow 1,
\end{equation}
for some central, diagonalizable subgroup scheme $Z'$ (see \cite[II.1.6]{jantzen}).
Since $G'_\bk$ is geometrically standard, then $Z'_{\bk}$ is both smooth by \cite[Definition 2.11 and Theorem 5.2]{her:scrg}, and finite since $G'_{\bk}$ is a finite product 
of quasi-simple group schemes. In particular, $Z'$ is constant. 

Let $G' = \prod_{i=1}^r G_i$, and let 
$x' = (x'_1,\dots, x'_r) \in \fg'_\O$ be a balanced nilpotent section,
and let $x = d\pi(x') \in \fg_\O$. 
Let $H'$ denote the centralizer of $x'$, and let $H$ denote the centralizer of $x$. By the same argument as in \cite[Lemma~2.33]{ls}, 
there exists an isogeny
\begin{equation}\label{eqn:semi-simple-intersect}
1\rightarrow Z'_{\ol{\bk}} \rightarrow H'_{\ol{\bk}} \rightarrow H_{\ol{\bk}} \rightarrow 1,
\end{equation}
induced by base changing \eqref{eqn:semi-simple-isogeny} to $\ol{\bk}$, and restricting down to the centralizer $H'_{\ol{\bk}}$. 
Now since $Z'_{\ol{\bk}}$ is smooth, and $H'_{\ol{\bk}}$ is smooth since $x'$ is balanced, then 
$H_{\ol{\bk}}$ must also be smooth. It follows that the section $x \in \fg_\O$ is also balanced. 

 Observe now that the isogeny sends $H'^\circ_{\ol{\bk}}$ onto $H^\circ_{\ol{\bk}}$, and thus induces a sequence 
\begin{equation*}
1 \rightarrow \frac{Z'_{\ol{\bk}}H'^\circ_{\ol{\bk}}}{H'^\circ_{\ol{\bk}}} \rightarrow A'(x_{\ol{\bk}}) \rightarrow A(x_{\ol{\bk}}) \rightarrow 1,
\end{equation*}
where we identify 
$\displaystyle \frac{Z'_{\ol{\bk}}H'^\circ_{\ol{\bk}}}{H'^\circ_{\ol{\bk}}} \cong \frac{Z'_{\ol{\bk}}}{Z'_{\ol{\bk}}\cap H'^\circ_{\ol{\bk}}}$. 

Without loss of generality, we can assume that  $x' \in \fg'_\O$ also satisfies the conditions of Lemma~\ref{lem:torus-intersect}.
Applying Proposition~\ref{prop:lift-torus} to $x'$ and $G'$ now provides a split, maximal torus $S' \subseteq H'$. 
As in the proof of Lemma~\ref{lem:simple-comp-group}, the proof of this lemma will follow by showing that 
$Z'\cap S'$ is both diagonalizable and constant. 
The former property holds by Proposition~\ref{prop:torus-intersect}, and the latter propery can be deduced from Lemma~\ref{lem:diag-constant} since 
$Z'$ is constant and diagonalizable, and hence any diagonalizable subgroup scheme of $Z'$ must also be constant. 
\end{proof}

The preceding argument can also be extended to arbitrary reductive groups $G$ with $G_\F$ geometrically standard. 
\begin{prop}\label{prop:comp-group}
Let $G$ be a reductive group (with $G_\F$ geometrically standard), then $|A(x_{\ol{\K}})|=|A(x_{\F})|$. 
\end{prop}
\begin{proof}
In this proof $\bk$ will denote either $\K$ or $\F$. 

Let $G' = [G,G]$, then by \cite[II.1.18]{jantzen} there exists an isogeny
\begin{equation}\label{eqn:reductive-isogeny}
1 \rightarrow T_1 \cap T_2 \rightarrow G' \times T_2 \xrightarrow{\pi} G \rightarrow 1,
\end{equation}
where $T_1$ and $T_2$ are tori and $T_1\cap T_2 \subseteq Z(G'\times T_2)$ is central, diagonalizable and finite. 
Let $Z' = T_1 \cap T_2$. The arguments appearing immediately below \eqref{eqn:semi-simple-isogeny}
 in the proof of Lemma~\ref{lem:semisimple-comp-group}, also imply that $Z'$ is diagonalizable and constant. 

Let $\mathfrak{h}_\O$ denote the Lie algebra of $T_2$, and note that there must exist a balanced 
nilpotent section $x \in \fg_\O$ which satisfies the conditions of 
Lemma~\ref{lem:torus-intersect}, and is of the form $x = d\pi(x')$ for some balanced nilpotent section 
 $x' \in \fg'_\O \times \mathfrak{h}_\O$, which also satisfies the conditions of Lemma~\ref{lem:torus-intersect}. 
 (This again follows from the arguments appearing immediately below
  \eqref{eqn:semi-simple-intersect}.)

Let $H' \subseteq G' \times T_2$ be the centralizer of $x'$, and let $H \subseteq G$ be the centralizer of $x$. Again, 
as in \cite[Lemma~2.33]{ls}, there is an isogeny 
\[
1 \rightarrow Z'_{\ol{\bk}} \rightarrow H'_{\ol{\bk}} \rightarrow H_{\ol{\bk}} \rightarrow 1,
\]
obtained from \eqref{eqn:reductive-isogeny} by base-changing to $\ol{\bk}$, and restricting down to $H'_{\ol{\bk}}$. 
Now just as in the proof of Lemma~\ref{lem:semisimple-comp-group}, it suffices to show that the finite groups
\[
 \frac{Z'_{\ol{\K}}{H'^\circ_{\ol{\K}}}}{H'^\circ_{\ol{\K}}} \cong \frac{Z'_{\ol{\K}}}{Z'_{\ol{\K}}\cap H'^\circ_{\ol{\K}}},
 \quad 
  \frac{{Z'_{\F}}{H'^\circ_{\F}}}{H'^\circ_{\F}} \cong \frac{Z'_{\F}}{Z'_{\F}\cap H'^\circ_{\F}}
\] 
have
the same order. However, this follows from the argument given in the last paragraph
of the proof of Lemma~\ref{lem:semisimple-comp-group}. 
\end{proof}
\begin{rmk}
In Theorem~\ref{thm:integral-comp-group}, it will be proven that the component groups are actually isomorphic (not just of the same order). 
The isomorphism of the component groups for general reductive $G$ with geometrically standard fiber $G_\F$ was originally claimed in \cite[Theorem B]{mcninch2007}, however an error was later found in the proof.
\end{rmk}

 \subsection{Proof of Theorem~\ref{thm:main-thm}}\label{subsec:disconnected-case}
In this subsection, we will establish the smoothness of $G^x$. 
The key step will be to show that  $G^x_{\text{tf},\F}$ contains all the connected components of $G^x_{\F}$. 
(We are still maintaining our assumption that $\F = \ol{\F}$.)

Label the connected components of $G^x_\F$ by $(G^x_\F)^i$ for $i=0,\dots, m-1$, where $m = |A(x_{\F})|$ and
$(G^x_\F)^0 := (G^x_\F)^\circ$ is the identity component subgroup scheme. 
Then the decomposition of $G^x_{\F}$ into its connected components induces the decomposition 
 \[
 \F[G^x_{\F}] = \F[(G^x_{\F})^{\circ}] \times \F[(G^x_{\F})^{1}] \times \cdots \times \F[(G^x_{\F})^{m-1}],
 \]
where $\F[(G^x_{\F})^{\circ}] \cong  \F[(G^x_{\F})^{i}]$ for all $i \geq 1$ (as algebras).
The fact that $G^x_{\text{tf},\F}$ is smooth implies
\[
 \F \otimes \O[G^x_{\text{tf}}] = \F[(G^x_{\F})^{\circ}] \times \F[(G^x_{\F})^{1}] \times \cdots \times \F[(G^x_{\F})^{k-1}],
\] 
where $1 \leq k \leq m$. In particular, 
$\F \otimes \O[G^x_{\text{tf}}] \subseteq \F[G^x_{\F}]$ and $G^x_{\text{tf},\F}$ has precisely $k$ connected components. 
Our goal is to show $k=m$. 

If necessary, let us enlarge $\O$, and consequently $\K$, by a finite extension\footnote{The integral closure of a (complete) DVR in a finite algebraic extension is a
 finitely-generated (complete) DVR (cf. \cite[Chap.~1, Proposition~8, and Chap.~2, Proposition~3]{ser:lf}).}
  (the assumption $\ol{\F} = \F$ implies that $\F$ remains unchanged), so that by Proposition~\ref{prop:comp-group}
\begin{equation}\label{eqn:K-rational-idempotents}
 \K[G^x_{\K}] = \K[(G^x_{\K})^{\circ}] \times \K[(G^x_{\K})^{1}] \times \cdots \times \K[(G^x_{\K})^{m-1}].
\end{equation}
This decomposition is given by a set of orthogonal idempotents 
$\epsilon_0,\epsilon_1,\dots, \epsilon_{m-1} \in \K[G^x_{\K}]$ with 
$\epsilon_i^2 = \epsilon_i$ for $i \geq 0$, $\epsilon_i\epsilon_j = 0$ for $i\neq j$ and 
${\epsilon_0 + \cdots + \epsilon_{m-1} = 1}$. Now
\[
\epsilon_i = \frac{h_i}{\omega^{n_{i}}}
\]
where $h_i \in \O[G^x_{\text{tf}}]$ and $n_i \geq 0$ for all $i$. Assume that each $n_i$ is chosen minimally 
so that $h_i \not\in \omega\O[G^x_{\text{tf}}]$. Thus,
\begin{equation}\label{eqn:square-h}
  \epsilon_i^2 = \epsilon_i \implies h_i^2 = \omega^{n_i}h_i.
\end{equation}
 On the other hand, $h_i \not\in \omega\O[G^x_{\text{tf}}]$ implies that 
 $\overline{h_i} \neq 0 \in \O[G^x_{\text{tf}}]/\omega\O[G^x_{\text{tf}}] \cong \F[G^x_{\text{tf}, \F}]$. But by \eqref{eqn:square-h}, 
 $\overline{h_i}^2 = 0$ if $n_i > 0$ (which contradicts the fact that $\F[G^x_{\text{tf}, \F}]$ has no nonzero nilpotent elements since it is reduced), so 
 $n_i =0$ for all $i$, and therefore, $\epsilon_i \in \O[G^x_{\text{tf}}]$ for all $i$.  
 
 This gives an internal decomposition 
 \[
 \O[G^x_{\text{tf}}] = A_0 \times A_1 \times \cdots A_{m-1},
 \]
 where $A_i = \epsilon_i\O[G^x_{\text{tf}}]$. Tensoring with $\F$ gives an internal algebra decomposition
 \[
  \F \otimes \O[G^x_{\text{tf}}] = \F\otimes A_0 \times \F\otimes A_1 \times \cdots \times \F\otimes A_{m-1},
 \] 
 since for all $i$,  $\overline{\epsilon_i} \neq 0$ and $\overline{\epsilon_i}^2 = \overline{\epsilon_i}$, also
 $\overline{\epsilon_i}\overline{\epsilon_j} = 0$ for $i \neq j$ and 
 ${\overline{\epsilon_0} + \cdots + \overline{\epsilon_{m-1}} = 1}$. Thus, $G^x_{\text{tf},\F}$ has at least $m$ 
 connected components, and hence $k=m$.  In particular, $G^x_{\text{tf}, \F} = G^x_{\F}$, and hence, $G^x_{\text{tf}} = G^x$ so 
 $G^x$ is smooth in this case. 
\begin{rmk}\label{rmk:large-DVR}
We have just shown that Theorem~\ref{thm:main-thm} will hold if $\O$ is enlarged by a DVR $\O \subset \O'$ where the residue field of $\O'$ is algebraically 
closed and the fraction field of $\O'$ is large enough so that
\eqref{eqn:K-rational-idempotents} holds. 
\end{rmk}

\begin{proof}[Proof of Theorem~\ref{thm:main-thm}]
For the proof, we return to the setup from \S\ref{sec:prelim}, where the only condition on $\O$ is that it is complete and $\F = \O/\omega$.

Let $\ol{\O}$ denote the completion of the maximal unramified extension of $\O$. By definition, $\omega$ is also the uniformizer for 
 $\ol{\O}$ and $\ol{\O}/ \omega\ol{\O} = \ol{\F}$. It also possible to enlarge $\ol{\O}$ by a \emph{finite} purely ramified integral extension 
 $\O' \supseteq \ol{\O}$  (cf. \cite[\S 1.4]{ser:lf}), such that \eqref{eqn:K-rational-idempotents} holds over the fraction field $\K'$ of $\O'$, where we note $\O'$ must 
 have the same residue field as $\ol{\O}$. 
 By Remark~\ref{rmk:large-DVR},  the base change $G^x_{\O'}$ is smooth, and in particular,
 \[
 \O'[G^x_{\O'}] = \O'\otimes \O[G^x]
 \]
 is torsion-free. Observe that $\O'$ is torsion-free as an $\O$-module, since $\O$ and $\O'$ are integral domains and $\O \subseteq \O'$. The fact that 
 $\O$ is a DVR now implies $\O'$ is also a flat $\O$-module (since these properties are equivalent for discrete valuation rings). Thus, 
 the functor $\O' \otimes_\O -$ is exact, and the natural map
 \begin{equation}\label{eqn:induced-base-change}
 \O[G^x] \xrightarrow{1\otimes \id} \O'\otimes \O[G^x]
 \end{equation}
 is injective. To see why \eqref{eqn:induced-base-change} is injective, first let  $0 \neq f \in \O[G^x]$, be arbitrary.  There is a commutative diagram 
 \[
\begin{CD}
     \O f    @>>>   \O[G^x]   \\
         @VVV           @VV 1\otimes \id V                                           \\ 
     \O' \otimes \O f    @>>>    \O'\otimes \O[G^x], 
\end{CD}
\]
where the bottom map is injective by exactness of $\O'\otimes_{\O}-$. 
 
 Now suppose that $f$ is torsion, then since $\O f$ is cyclic and non-zero, we must have  
$\O f \cong \O/\omega^k$ for some $k \geq 1$. If $\omega'$ is the uniformizer of 
$\O'$, then $(\omega')^r = \omega$ for some $r \geq 1$. Thus
\[
\O'\otimes \O f \cong \O'/\omega^r \cong \O'/(\omega')^{rk},
\]
where $rk \geq 1$, and so must be non-zero. The injectivity of the bottom map in the preceding diagram implies $\O' (1\otimes f) \cong \O'/(\omega')^{rk}$,
therefore $1 \otimes f$ must also be non-zero, and torsion. But this is not possible since $\O'\otimes \O[G^x]$ is torsion-free.

This allows us to identify identify $\O[G^x] \subseteq \O'\otimes \O[G^x]$, which implies that $\O[G^x]$ must be torsion-free, since $\O'\otimes \O[G^x]$ is 
 torsion-free and any non-zero $\O$-submodule of a torsion-free module is torsion-free. 
%
\end{proof}

\subsection{The component group scheme}\label{ssec:comp-grp-scheme}
We will continue to maintain our assumption that $\F = \ol{\F}$. 

By \cite[Definition~13.12]{mil:ALA}, for \emph{any} field $\bk$ over $\O$, there exists a \emph{component group scheme}, denoted 
$A_{\bk}(x)$. This is defined to be the spectrum of the largest \'etale subalgebra of $\bk[G^x_{\bk}]$ (cf. \cite[13b]{mil:ALA}).  Moreover, the coordinate algebra $\bk[A_{\bk}(x)]$ 
is naturally a Hopf subalgebra of $\bk[G^x_{\bk}]$. 
The inclusion $\bk[A_{\bk}(x)] \hookrightarrow \bk[G^x_{\bk}]$, induces a surjective group scheme homomorphism $G^x \rightarrow A_{\bk}(x)$. 
The \emph{connected component}
 of $G^x_\bk$, denoted $(G^x_\bk)^\circ$, can now be defined as the normal subgroup scheme given by the kernel of this morphism. 
 
It follows from \cite[Proposition~13.18]{mil:ALA}, that for any field extension $\bk' \supseteq \bk$, 
\[
A_{\bk'}(x) \cong A_{\bk}(x)_{\bk'}, \quad \text{and} \quad (G^x_{\bk'})^\circ \cong (G^x_{\bk})^\circ_{\bk'}.
\]
In particular, $(A_{\K}(x)_{\ol{\K}})(\ol{\K}) \cong A(x_{\ol{\K}})$ and $(A_{\F}(x))(\F) \cong A(x_\F)$, 
where the discrete groups on the right-hand-side were considered in \S\ref{subsec:comp-fibers}.

In \S\ref{subsec:disconnected-case}, it was shown that if $x \in \fg_\O$ is a balanced nilpotent section, and $\K$ satisfies \eqref{eqn:K-rational-idempotents}, then 
there exist orthogonal idempotents $\epsilon_0,\dots,\epsilon_{m-1} \in \O[G^x]$ 
so that
\begin{equation}\label{eqn:O-idempotent}
\O[G^x] = \epsilon_0\O[(G^x)] \times \epsilon_1\O[(G^x)] \times \cdots \times \epsilon_{m-1}\O[(G^x)],
\end{equation}
where $m = |A(x)_{\ol{\K}}| = |A(x)_{\F}|$. 

Therefore, the $\O$-scheme $G^x$ has a decomposition
\begin{equation}\label{eqn:component-decomp}
G^x= (G^x)^\circ \sqcup (G^x)^1 \sqcup \cdots \sqcup (G^x)^{m-1},
\end{equation}
where $(G^x)^i = \Spec(\epsilon_i\O[G^x])$ and $(G^x)^i_{\ol{\K}}$ and $(G^x)^i_\F$ give the complete set of connected components for $G^x_{\ol{\K}}$ and 
$G^x_{\F}$ respectively. 

\begin{thm}\label{thm:integral-comp-group}
Let $x \in \fg_\O$ be a balanced nilpotent section, and suppose that $\O$ is such that $\F = \ol{\F}$ and $\K$ satisfies \eqref{eqn:K-rational-idempotents}.
If we let $A(x)$ be the constant scheme defined by the subalgebra 
\begin{equation}\label{eqn:comp-coord-algebra}
\O[A(x)] := \sum_{i=0}^{m-1} \O\epsilon_i \subseteq \O[G^x],
\end{equation}
where $\epsilon_0,\dots,\epsilon_{m-1} \in \O[G^x]$ are as in \eqref{eqn:O-idempotent}, then 
\begin{enumerate}
\item for any field $\bk$ over $\O$, $A(x)_{\bk} \cong A_{\bk}(x)$, 
\item $\O[A(x)]$ is a Hopf subalgebra of $\O[G^x]$,
\item  $(G^x)^\circ$ is the kernel of the induced homomorphism $G^x \twoheadrightarrow A(x)$, so that $A(x) \cong G^x/(G^x)^\circ$ and 
$(G^x)^\circ \unlhd G^x$ is a normal subgroup scheme. 
\end{enumerate}

\end{thm}
\begin{proof}

We begin by proving (1). Let us first consider the case where $\bk = \K$, then it suffices to show that
$\K[A(x)_\K] = \K\otimes \O[A(x)] \subseteq \K[G^x_\bk]$ is the largest \'etale subalgebra of $\K[G^x_\K]$. To see why this 
is the case, let $R \subseteq \bk[G^x_\K]$ be the maximal \'etale subalgebra,  so that by definition, $R \supseteq \K[A(x)_\K]$ (such a subalgebra always exists). 
By \cite[Proposition~13.8]{mil:ALA}, $\ol{\K}\otimes R \subseteq \ol{\K}[G^x_{\ol{\K}}]$ also gives the maximal \'etale subalgebra. Now, \cite[Corollary~13.9]{mil:ALA}
and \cite[Lemma~13.4]{mil:ALA} imply
\[
\dim_\K R = \dim_{\ol{\K}}(\ol{\K}\otimes R) = |\{\text{connected components of $G^x_{\ol{\K}}$}\}| = m,
\]
where the rightmost equality follows from assumption \eqref{eqn:K-rational-idempotents}. Thus, $\dim_\K R = \dim_\K \K[A_\K(x)]$, and therefore, 
$R = \K[A_\K(x)]$. 
The same argument also show that $\F[A(x)_\F]$ is the maximal \'etale subalgebra of $\F[G^x_\F]$. 
Finally, suppose $\bk$ is any \emph{field} over $\O$. Then either $\bk \supseteq \K$ or $\bk \supseteq \F$, and in both cases 
\cite[Proposition~13.8]{mil:ALA} implies that $\bk[A(x)_\bk]$ is the maximal \'etale subalgebra of $\bk[G^x_\bk]$. So we have verified (1). 

Now we will verify (2). Let $\Delta$, $\sigma$ and $\varepsilon$ denote the coproduct, antipode and counit for $\O[G^x]$ respectively. 
It suffices to show 
\[
\Delta(\O[A(x)]) \subseteq \O[A(x)]\otimes \O[A(x)],
\]
$\sigma(\O[A(x)]) \subseteq \O[A(x)]$ and $\varepsilon(\O[A(x)]) = \O$. 
The third identity can be verified immediately by observing that $\varepsilon(\O[A(x)]) \subseteq \O$ is an $\O$-subalgebra of $\O$, which implies 
equality since $\O$ cannot have any proper $\O$-subalgebras. 

To verify the first two identities, we first note that for an $\O$-algebra $\bk$, the Hopf algebra structure of $\bk[G^x_\bk] = \bk\otimes \O[G^x]$ is given by 
$\Delta_\bk := \id_\bk \otimes \Delta$, $\varepsilon_\bk := \id_\bk \otimes \varepsilon$, 
$\sigma_\bk := \id_\bk \otimes \sigma$. So, in particular, for $\bk = \K$, the first two identities must hold for $\K\otimes \O[A(x)]$ by  (1). 
To verify the second identity, it suffices to show $S(\epsilon_i) \in \O[A(x)]$ for all $i$ since $S$ is $\O$-linear. However, the fact that 
$\O[A(x)]$ is torsion-free, allows us to identify $\O[A(x)] \subset \K \otimes \O[A(x)]$, so that
$\sigma = (\sigma_\K)|_{\O[A(x)]}$. Thus, if $i$ is arbitrary, then 
\[
\sigma(\epsilon_i) = \sigma_\K(\epsilon_i)  = a_0\epsilon_0 + a_1 \epsilon_1 + \cdots + a_{m-1}\epsilon_{m-1} \in (\K\otimes \O[A(x)])\cap \sigma(\O[A(x)])
\]
for some $a_0,\dots, a_{m-1} \in \K$. But since $\sigma$ is a morphism of $\O$-algebras, and the $\epsilon_0, \dots, \epsilon_{m-1}$ are pairwise orthogonal
idempotents, then 
\[
\sigma(\epsilon_i)^2 = \sigma(\epsilon_i^2) = \sigma(\epsilon_i)  
\]
implies 
\[
a_0^2\epsilon_0 + a_1^2\epsilon_1 + \cdots + a_{m-1}^2\epsilon_{m-1} = a_0\epsilon_0 + a_1\epsilon_1 + \cdots + a_{m-1}\epsilon_{m-1}, 
\]
so $a_i^2 = a_i$, and hence, $a_i \in \{0,1\} \subset \O$ for all $i$. 
So the
second identity is verified. 

Similarly, if $i$ is arbitrary, then 
\[
\Delta(\epsilon_i) = \Delta_\K(\epsilon_i) = \sum_{(j,k)} a_{jk}\epsilon_j\otimes \epsilon_k  \in \left(\K\otimes \O[A(x)]\otimes \K\otimes \O[A(x)]\right)  \cap \Delta(\O[A(x)])
\]
for some $a_{jk} \in \K$. Now observe set of $\epsilon_j \otimes \epsilon_k$ gives a linearly independent
 set of pairwise orthogonal idempotents for $\O[G^x] \otimes \O[G^x]$, 
and that $\Delta$ is a morphism algebras. Thus, 
\[
\Delta(\epsilon_i)^2 = \Delta(\epsilon_i^2) = \Delta(\epsilon_i)
\]
implies 
\[
\sum_{(j,k)} a_{jk}^2\epsilon_j\otimes \epsilon_k  = \sum_{(j,k)} a_{jk}\epsilon_j\otimes \epsilon_k,
\]
so that $a_{jk}^2 = a_{jk}$, which forces $a_{jk} \in \{0,1\} \subset \O$. 
Therefore, $\O[A(x)]$ is a Hopf subalgebra of $\O[G^x]$. 

So the inclusion $\O[A(x)] \hookrightarrow \O[G^x]$ of Hopf algebras, induces a surjective map of group schemes
\[
G^x \twoheadrightarrow A(x).
\]
Finally, let $H$ be the kernel of this homomorphism. From the definition of the kernel of a group scheme homomorphism, we have 
\[
\O[H] = \O[G^x]/(\epsilon_1+\cdots+\epsilon_{m-1})\O[G^x] = \O[(G^x)^\circ].
\]
Thus, $H = (G^x)^\circ$, and in particular, $(G^x)^\circ$ is a normal subgroup scheme of $G^x$.

\end{proof}

Now we can prove Theorem~\ref{thm:mcninch-sommers-generalization}. 
\begin{proof}[Prooof of Theorem~\ref{thm:mcninch-sommers-generalization}]
Let us now return to our hypothesis on $\O$, $\K$ and $\F$ from \S\ref{sec:prelim}. Just as in the proof of Theorem~\ref{thm:main-thm}, we begin by replacing 
$\O$ with the completion of its maximal unramified extension, which we denote by $\ol{\O}$. Now the residue field of $\ol{\O}$ is the algebraic closure $\ol{\F}$ 
of $\F$, and let $\K'' \supseteq \K$ be the fraction field of $\ol{\O}$. Also let $\O'$ be a finite integral extension of $\ol{\O}$ 
(with fraction field $\K' \supseteq \K'' \supseteq \K$) such that the hypothesis of Theorem~\ref{thm:integral-comp-group} 
is satisfied. 

Applying Theorem~\ref{thm:integral-comp-group} to this setup, immediately implies $A(x_{\ol{\K'}}) \cong A(x_{\ol{\F}})$, where $\ol{\K'}$ denotes the 
algebraic closure of $\ol{\K'}$. Finally,  $A_{\ol{\K}}$ denote the component group scheme, and noting that $\ol{\K} \subseteq \ol{\K'}$, then it follows from 
\cite[Proposition~13.8]{mil:ALA} that $A(x_{\ol{\K}})\cong A(x_{\ol{\K'}})$ as groups. Therefore, $A(x_{\ol{\K}}) \cong A(x_{\ol{\F}})$. 
\end{proof}


%


\section{Centralizers for the $G\times \Gm$ action}\label{sec:graded-action}
\subsection{Smoothness in the graded case}
For simplicity, we will introduce the notation $\mathcal{G} = G \times \Gm$. 
The scheme $\fg$ is also equipped with a
$\mathcal{G}$ action, where $\Gm$ acts by the \emph{cohomological action},
${
t\cdot x = t^{-2}x 
}$
for $t \in \Gm(\bk)$ and $x \in \fg(\bk)$ and any $\O$-algebra $\bk$. 
This gives $\O[\fg]$ a non-negative, even grading which is generated in degree 2. Now, for any balanced nilpotent $x \in \fg_\O$,\footnote{Recall that $\fg_\O := \fg(\O)$.}
consider the centralizer $\mathcal{G}^x$, as well as the centralizers for the base-changes 
$\mathcal{G}_{\K}^x$ and $\mathcal{G}_{\F}^x$.

%
Suppose now that $x \in \fg_{\O}$ is a balanced nilpotent section, then by
 \cite[Theorem 1.2.1]{mcninch2016}, there exists an integral \emph{associated cocharacter} 
\[
\phi_x: \Gm \longrightarrow G^x,
\]
such that $\phi_{x,\K}$ and $\phi_{x,\F}$ are the associated cocharacters arising from the Jacobson-Morozov triples for $x_{\K}$ and $x_{\F}$ respectively.

Let ${\text{Int}: (G^x)^{op} \longrightarrow \text{Aut}(G^x)}$, with $\text{Int}_h(g) = hgh^{-1}$ for $h, g \in G^x(\bk)$ be a
 morphism of group schemes. 
An action of $\Gm$ on $G^x$ can then be given by
\begin{equation}\label{eqn:graded-action}
t \cdot g = \phi_x(t^{-1})g \phi_x(t) = \text{Int}_{\phi_x(t^{-1})}(g),
\end{equation}
for $g \in G^x(\bk)$ and any $\O$-algebra $\bk$.  
Let 
\[
\Gm \ltimes_{\phi_x} G^x,
\]
be the semi-direct product formed from this action. 
\begin{prop}\label{lem:integral-scaling}
There is an isomorphism of group schemes
$$
\Gm \ltimes_{\phi_x} G^x \xrightarrow{\sim} \cG^x \subseteq G \times \Gm,
$$
given by $t \ltimes g \mapsto (g \phi_x(t^{-1}), t^{-1})$ for any $g \in G(\bk)$, $t \in \Gm(\bk)$ and any $\O$-algebra $\bk$. 
\end{prop}
\begin{proof}
Let us first note that this map is well-defined and natural for any $\O$-algebra $\bk$. 
Now, letting $\bk$ be arbitrary and identifying $x \in \fg(\bk)$ with its image under the morphism $\fg(\O) \rightarrow \fg(\bk)$,  
we observe that  $(g,t) \in \cG^x(\bk) \subseteq (G \times \Gm)(\bk)$ if and only if
\[
x = (g,t)\cdot x = \text{Ad}_{(1,t)} \circ \text{Ad}_{(g,1)}x = t^{-2}\text{Ad}_gx,
\]
or equivalently,
${
\text{Ad}_{g}x = t^{2}x.
}$
Now since $\phi_x$ is an associated cocharacter, then $x \in \fg_2$ with respect
to its induced grading on $\fg$ (equivalently ${\text{Ad}_{\phi_x(t)}x = t^2x}$ for any $t \in \Gm(\bk)$). Thus, 
\[
\text{Ad}_{g}x = \text{Ad}_{\phi_x(t)}x \iff \text{Ad}_{g\phi_x(t^{-1})}x = x,
\]
and so $g\phi_x(t^{-1}) \in G^x(\bk)$. Therefore, 
\[
 (g,t) \in \cG^x(\bk) \iff (g,t) = (h\phi_x(t),t) \text{ for $h = g\phi_x(t^{-1}) \in G^x(\bk)$}.
\]
Therefore, this map is canonically a group isomorphism for every $\O$-algebra $\bk$, and hence is an isomorphism of group schemes. 
\end{proof}

\begin{cor}\label{cor:smoothness-field}
The scheme-theoretic centralizers $\mathcal{G}_{\K}^x$ and $\mathcal{G}_{\F}^x$ are smooth. 
\end{cor}

Now recall that if $\bk$ is a field over $\O$, then there exists an additional Levi-decomposition of $G^x_{\bk}$. 
Let $G^{x}_{\bk, \text{red}} \times \Gm$ act on $G^{x}_{\bk, \text{unip}}$ by 
\[
(g,t)\cdot u = g\phi_{x}(t^{-1})u\phi_x(t)g^{-1},
\]
for $(g,t) \in (G^{x}_{red,\bk}\times \Gm)(\bk')$, $u \in G^{x}_{\bk, \text{unip}}(\bk')$ and any $\bk$-algebra $\bk'$. 
Let \\
$(G^{x}_{\bk,\text{red}} \times \Gm) \ltimes G^{x}_{\bk, \text{unip}}$ be the semi-direct product formed from this action. 

\begin{cor}\label{lem:scaling-centralizer}
For $\bk \in \{\K, \F\}$, there is an isomorphism
$$
(G^{x}_{\bk,\text{red}} \times \Gm) \ltimes G^{x}_{\bk, \text{unip}} \xrightarrow{\sim} \mathcal{G}_{\bk}^{x} 
$$
given by
$(g,t) \ltimes u \mapsto (g\phi_{x}(t^{-1})u, t^{-1})$ for 
${(g,t)\ltimes u \in \big((G^{x}_{\bk,\text{red}} \times \Gm)\ltimes G^{x}_{\bk,\text{unip}}\big)(\bk')}$ and any $\bk$-algebra $\bk'$. 
In particular, 
${\cG^x_{\bk,\text{red}} \cong G^{x}_{\bk,\text{red}} \times \Gm }$ and 
${\cG^x_{\bk,\text{unip}} \cong G^{x}_{\bk,\text{unip}} }$.
\end{cor}
\begin{proof}
This follows by observing that for any $\bk$-algebra $\bk'$, this map is canonically equivalent to the one in Proposition~\ref{lem:integral-scaling},
and is therefore an isomorphism.
\end{proof}


The following theorem is now immediate. 
 \begin{thm}\label{thm:graded-analogues}
Let $x \in \fg_\O$ be a balanced nilpotent section, and let $\bk \in \{\K, \F\}$, then
the morphism $\mathcal{G}^x \rightarrow \Spec(\O)$ is smooth and 
$
\mathcal{G}^x_{\ol{\bk}}/(\mathcal{G}^x_{\ol{\bk}})^\circ (\ol{\bk}) \cong  A(x_{\ol{\bk}}).
$

Furthermore, if the conditions of Theorem~\ref{thm:integral-comp-group} are satisfied, then 
the morphism $G^x \twoheadrightarrow A(x)$ lifts to a group scheme homomorphism 
$\mathcal{G}^x \twoheadrightarrow A(x)$, where the base changes $A(x)_\bk$ also give the component group
schemes for $\mathcal{G}^x_{\bk}$ (i.e. there are group scheme isomorphisms
$
\mathcal{G}_\bk^x/(\mathcal{G}_\bk^x)^\circ \cong G^x_\bk/(G^x_\bk)^\circ.)
$
\end{thm}



\subsection{Integral lattices}
For a balanced nilpotent section $x \in \fg_{\O}$, let $H \in \{G^x, \cG^x\}$, then the flatness of $H$ now implies that for any finite-dimensional 
$H_{\K}$-module
$V$, there exists an $H$-stable $\O$-lattice $M\subset V$ (cf. \cite[I.10.4]{jantzen}).  

\begin{lem}\label{lem:grading-unipotent}
Let $\bk \in \{\K, \F\}$. If $N$ is a $\cG^x_{\bk} = \Gm \ltimes_{\phi_x} G^x_{\bk}$-module such that $(t \ltimes g)\cdot n = (1\ltimes g) \cdot n$ for any
$t \ltimes g \in (\Gm \ltimes_{\phi_x} G^x_{\bk})(\bk')$, $n \in N(\bk')$ and any $\bk$-algebra $\bk'$, then $G^x_{\bk,\text{unip}}$ acts trivially on $N$. 
\end{lem}
\begin{proof}
If we assume the hypothesis, then since $\cG^x_{\bk}$, $\Gm$ and $G^x_{\bk}$ are all reduced by Corollary~\ref{cor:smoothness-field}, 
it suffices to show that
$G^x_{\bk, \text{unip}}(\overline{\bk})$ acts trivially on $N(\overline{\bk}) = N \otimes_{\bk}\overline{\bk}$. 
So without loss of generality, assume that $\bk = \overline{\bk}$, so that it suffices to work with the
geometric points. 

The restriction of $N$ to $G^x_{\bk}$ corresponds to a group variety homomorphism 
\[
\psi: G^x_{\bk} \longrightarrow GL(N),
\]
and our goal is show $\psi(G^x_{\bk,\text{unip}}) = \{1\}$. However, since 
\[
\text{Lie}\, \psi(G^x_{\bk,\text{unip}})  = d\psi\, (\text{Lie}\, G^x_{\bk, \text{unip}}),
\]
then it suffices to show ${d\psi\, (\text{Lie}\, G^x_{\bk,\text{unip}}) = 0}$.

The fact that $N$ is a $(\Gm \ltimes_{\phi_x} G^x_{\bk})$-module is equivalent to saying that 
$\psi$ is $\Gm$-equivariant, where $\Gm \curvearrowright G^x_{\bk}$ via \eqref{eqn:graded-action} (e.g. by 
$\text{Int}_{\phi_x(t^{-1})}$), and that
$\Gm \curvearrowright GL(N)$ trivially by the hypothesis. Moreover, there are induced actions of $\Gm$ on the respective Lie 
algebras (given by the adjoint action), so that 
\[
d\psi: \fg^x_{\bk} \longrightarrow \mathfrak{g}\mathfrak{l}(N)
\]
is also $\Gm$-equivariant. Equivalent the respective Lie algebras are given gradings
\[
\fg^x_{\bk} = \bigoplus_{k \in \Z} (\fg^x_{\bk})_k, \quad \mathfrak{g}\mathfrak{l}(N) = \bigoplus_{k \in \Z} \mathfrak{g}\mathfrak{l}(N)_k,
\]
where the $\Gm$-equivariance of $d\psi$ induces a decomposition
${d\psi = \bigoplus_{k \in \Z} d\psi_k}$ such that
\[
d\psi_k: (\fg^x_{\bk})_k\longrightarrow \mathfrak{g}\mathfrak{l}(N)_k.
\]
Observe now that 
${\mathfrak{g}\mathfrak{l}(N) = \mathfrak{g}\mathfrak{l}(N)_0}$ and $d\psi_k =0$ for all $k \neq 0$,
 since the action of $\Gm$ on $GL(N)$, and hence on $ \mathfrak{g}\mathfrak{l}(N)$, is trivial.
By \eqref{eqn:graded-action}, the action of any  $t \in \Gm$ is given by the automorphism
\[
d(\text{Int}_{\phi_x}(t^{-1})) = \text{Ad}_{\phi_x(t^{-1})}: \fg^x_{\bk} \longrightarrow \fg^x_{\bk},
\] 
and thus by \cite[Proposition 5.10]{jantzen-nilp}, 
\[
\text{Lie}(G^x_{\bk, \text{red}}) = (\fg^x_{\bk})_0, \quad \text{Lie}(G^x_{\bk,\text{unip}}) = \bigoplus_{k < 0} (\fg^x_{\bk})_k,
\]
with ${ (\fg^x_{\bk})_k = 0 }$ for all $k > 0$. Finally, we see that since $d\psi_k = 0$ for all $k \neq 0$, then
$d\psi(\text{Lie}G^x_{\bk,\text{unip}}) = 0$. 
\end{proof}
\begin{rmk}
The reason for the sign difference between the $\Z$-grading on $\fg^x_{\bk}$ in the preceding proof, and the 
$\Z$-grading in  \cite[Proposition 5.10]{jantzen-nilp}  is due to the fact that the grading here is induced from the action \eqref{eqn:graded-action}, which gives $\text{Ad}_{\phi_x(t^{-1})}$, while
the cited proposition is with respect to the inverse action given by $\text{Ad}_{\phi_x(t)}$ for $t \in \Gm$. 
\end{rmk}

%
\begin{prop}\label{prop:reductive-lattice}
If $V$ is any $G^x_{\K}$-module which factors through $G^x_{\K,\text{red}}$, then there exists a 
$G^x$-stable $\O$-lattice $M \subset V$ such that 
$M_{\F}$ factors through $G^x_{\F,\text{red}}$
(i.e. $G^x_{\F,\text{unip}}$ acts trivially on $M_{\F}$). 
\end{prop}
\begin{proof}
 By Corollary~\ref{lem:scaling-centralizer}, $V$ can be lifted to a module for 
$\cG^x_{\K}$ which has trivial $\Gm$ and $G^x_{\K,\text{unip}}$ actions. Comparing with
 Proposition~\ref{lem:integral-scaling}, we observe that the restriction of $V$ to 
 \[
 (\Gm \ltimes_{\phi_x} 1)_{\K} \subseteq (\Gm \ltimes_{\phi_x} G^x_{\K})_{\K}
 \]
 is also trivial. 
   Now Theorem~\ref{thm:graded-analogues} ensures that $\cG^x$ is smooth (and hence flat over $\O$), therefore there must exist
a $\cG^x$-stable $\O$-lattice $M \subset V$, and  by Proposition~\ref{lem:integral-scaling}, it follows that $M$ has the structure of a 
$\Gm$-module by restricting to $\Gm \ltimes_{\phi_x} 1$. 
This structure is the equivalent to giving $M$ 
a grading $M = \bigoplus_{k \in \Z} M_k$. And moreover, 
\[
V = M \otimes \K = \bigoplus_{k \in \Z} (M_k\otimes \K),
\]
where $V_k = M_k\otimes \K$ is the grading arising from the $(\Gm \ltimes_{\phi_x} 1)_{\K}$ module structure. 
Since this module structure is trivial, then it must be the case that $V = V_0$ which, by the fact that $M$ is free, implies 
$M_k = 0$ for all $k \neq 0$. Finally, by base-changing to $\F$, it can be observed that the $(\Gm \ltimes_{\phi_x} 1)_{\F}$
module structure on $M_{\F}$ is given by 
\[
M_{\F} = M \otimes \K = \bigoplus_{k \in \Z} (M_k\otimes \F) = M_0 \otimes \F.
\]
This implies that $M_{\F}$ satisfies the conditions of Lemma~\ref{lem:grading-unipotent}, and therefore, it factors through 
$G^x_{\F,\text{red}}$. 
\end{proof}

%

\end{document}